\documentclass{elsarticle}
\usepackage{amsmath}
\usepackage{amssymb}
\usepackage{amsfonts}
\usepackage{rotating}
\usepackage{graphicx}
\usepackage{floatflt,epsfig}
\usepackage{lineno,hyperref}
\usepackage{enumerate}
\usepackage{colortbl}
\usepackage{array,tabularx,tabulary,booktabs}
\usepackage{longtable}
\usepackage{multirow}
\usepackage{wrapfig}
\usepackage{subcaption}
\newcolumntype{^}{>{\currentrowstyle}}

\journal{Discrete Mathematics}
\newtheorem{remark}{Remark}

\newtheorem{theorem}{Theorem}

\newtheorem{construction}{Construction}

\begin{document}
\renewcommand{\abstractname}{Abstract}
\renewcommand{\refname}{References}
\renewcommand{\tablename}{Table.}
\renewcommand{\arraystretch}{0.9}
\sloppy

\begin{frontmatter}
\title{New versions of the Wallis-Fon-Der-Flaass construction to create divisible design graphs}

\author[01]{Vladislav~V.~Kabanov}
\ead{vvk@imm.uran.ru}

\address[01]{Krasovskii Institute of Mathematics and Mechanics, S. Kovalevskaja st. 16, Yekaterinburg, 620990, Russia}

\begin{abstract}
A $k$-regular graph on $v$ vertices is a {\em divisible design graph}  with parameters $(v,k,\lambda_1 ,\lambda_2 ,m,n)$ if its vertex set can be partitioned into $m$ classes of size $n$, such that any two different vertices from the same class have  $\lambda_1$ common neighbours, and any two vertices from different classes have  $\lambda_2$ common neighbours whenever it is not complete or edgeless. If $m=1$, then a divisible design graph is strongly regular with parameters $(v,k,\lambda_1,\lambda_1)$. In this paper the Wallis-Fon-Der-Flaass construction of strongly regular graphs is modified to create new constructions of divisible design graphs. In some cases, these constructions lead to strongly regular graphs.
\end{abstract}

\begin{keyword}  Divisible design graph\sep Strongly regular graph

\vspace{\baselineskip}
\MSC[2020] 05B20\sep 05C50\sep 05E30
\end{keyword}
% 05C50 graphs and linear algebra (matrices, eigenvalues, etc.)
% 05C51 Graph designs and isomorphic decomposition
% 05B20 Combinatorial aspects of matrices
% 05E30 Association schemes, strongly regular graphs
% 15A18 eigenvalues, singular values, and eigenvectors
% 05B15 Orthogonal arrays, Latin squares, Room squares
\end{frontmatter}

\section{Introduction}

A $k$-regular graph on $v$ vertices is a {\em divisible design graph}  with parameters $(v,k,\lambda_1 ,\lambda_2 ,m,n)$ if its vertex set can be partitioned into $m$ classes of size $n$, such that any two different vertices from the same class have  $\lambda_1$ common neighbours, and any two vertices from different classes have  $\lambda_2$ common neighbours whenever it is not complete or edgeless.
The partition of a divisible design graph into  classes is called  a {\em canonical partition}.
If $m=1$, then a divisible design graph is a strongly regular with parameters $(v,k,\lambda_1,\lambda_1)$. In this paper the Wallis-Fon-Der-Flaass construction of strongly regular graphs is modified to create new constructions of divisible design graphs.
In some cases, these constructions lead to strongly regular graphs.

There is a connection between divisible design graphs and divisible designs since their vertices  can be considered as points, and the neighbourhoods of the vertices as blocks. Obviously, such a divisible design is symmetric, since it has a symmetric
incidence matrix with zero diagonal. R.C. Bose and W.S. Connor studied the combinatorial properties of divisible designs in~\cite{BC}. 

Divisible design graphs were first introduced by W.H.~Haemers, H.~Kharaghani and M.~Meulenberg in~\cite{HKM}.  In particular, the authors have proposed twenty constructions of divisible design graphs using various combinatorial structures. Some new divisible design graphs were constructed by D.~Crnkovi\'c and W.H.~Haemers in~\cite{CH}. 

The spectrum of any divisible design graph with parameters $(v,k,\lambda_1,\lambda_2,m,n)$ can be calculated using its parameters as follows:
$$\{k^1,\sqrt{k-\lambda_1}^{f_1},-\sqrt{k-\lambda_1}^{f_2},
\sqrt{k^2-\lambda_2 v}^{g_1},-\sqrt{k^2-\lambda_2 v}^{g_2}\},$$
where exponents are eigenvalue multiplicities. 
Moreover, $f_1+f_2 =  m(n - 1)$ and $g_1+g_2 = m-1$ \cite[Lemma 2.1]{HKM}.
The eigenvalues for all divisible design graphs in this article are calculated using this result.

W.D.~Wallis proposed in~\cite{W} a new construction of strongly regular graphs based on an affine designs and a Steiner 2-design. Later D.G.~Fon-Der-Flaass found how to modify a partial case of Wallis construction, when the corresponding Steiner $2$-design has blocks of size $2$,  in order to obtain hyperexponentially many strongly regular graphs with the same parameters \cite{FF}.
M.~Muzychuk in~\cite{MM} showed how to modify Fon-Der-Flaass ideas in order to cover all the cases of Wallis construction. Moreover, he showed that a Steiner 2-design in the original Wallis construction can be replaced by a partial linear space and  discovered new prolific constructions of strongly regular graphs. Some of the ideas of the above constructions turned out to be useful for new constructions of divisible design graphs as well.
Divisible design graphs are constructed in this paper by modification of the Wallis-Fon-Der-Flaass construction of strongly regular graphs. 

Just like W.D.~Wallis and D.G.~Fon-Der-Flaass we use affine designs for our construction. 

{\em  An affine design} $\mathcal{D}=(\mathcal{P}, \mathcal{B})$  with parameters $q$ and $r$ is a design, where $\mathcal{P}$ is a set of points and $\mathcal{B}$ is a set of blocks,  with the  following two properties:
\begin{itemize}
     \item[$(i)$] every two blocks are either disjoint or intersect in $r$  points;
     \item[$(ii)$] each block together with all blocks disjoint from it forms a parallel class: a set of $q$ mutually disjoint blocks partitioning all points of the design.
\end{itemize}

If $\mathcal{D}=(\mathcal{P}, \mathcal{B})$ is an affine design with parameters $q$ and $r$, then the number $\varepsilon:= (r-1)/(q-1)$ is integer and all parameters of $D$ are expressed in terms of $q$ and $r$ \cite[Lemma~1]{FF}:
   \begin{center}
    \begin{tabular}{cll}
$v$ & $q^2 r$ & the number of points;\\
$b$ & $q^3 \varepsilon + q^2 + q$ & the number of blocks;\\
$m$ & $q^2 \varepsilon + q + 1$ & the number of parallel classes;\\ 
$k$ & $q r$ & the block size;\\
$\lambda$ & $q \varepsilon + 1$ & the number of blocks containing \\
 & &   any 2 distinct points.\\
    \end{tabular}
  \end{center}

Any $d$-dimensional affine space over a finite field of order $q$  with all hyperplanes in the form of blocks is an affine design with $r=q^{d-2}$. This design has $q^d$ points, any block contains $q^{d-1}$ points,  $q^{d-2}$ points are on the intersection of any two different blocks, and the number of blocks containing any 2 distinct points is $(q^{d-1}-1)/(q-1)$. Other known examples are Hadamard $3$-designs, where~$q = 2$.

{\em A Latin square} of side $m$ is an $m\times m$ array in which each cell contains a single symbol from an $m$-set, such that each symbol occurs exactly once in each row and exactly once in each column. 
 A Latin square $\mathcal{L}=(e(i,j))_{m\times m}$ of side $m$ is {\em symmetric} if $e(i, j) = e(j, i)$ for all $i, j\in [m]$. A Latin square of side $m$ on the integers set $[m]:=\{1, 2,\dots ,m\}$ is {\em reduced} if in the first row and column the integers occur in natural order. Latin square of side $m$ is equivalent to
the multiplication table (Cayley table) of a quasigroup on $m$ elements, if we introduce a boundary row and a column. In our constructions, we use symmetric Latin squares instead of Steiner 2-designs.

We only consider undirected graphs without loops or multiple edges. For all necessary information about graphs we refer to \cite{BH}. For a survey of designs and Latin squares, we refer to \cite{CDW}

\section{First construction} 

Let $\mathcal{D}_1, \dots ,\mathcal{D}_m$ be arbitrary affine designs all with parameters $(q,q^{d-2})$, where $m=(q^d - 1)/(q-1)$ is the number of parallel classes of blocks in each $\mathcal{D}_i$. For all  $i\in [m]$, let $\mathcal{D}_i=(\mathcal{P}_i, \mathcal{B}_i)$. Parallel classes in each $\mathcal{D}_i$ are enumerated by integers from $[m]$ and
 $j$-th parallel class of $\mathcal{D}_i$ is denoted  by $\mathcal{B}_i^j$.
For any $x\in \mathcal{P}_i$, the block in the parallel class $\mathcal{B}_i^j$ which contains $x$ is denoted by $B_i^j(x)$.
\smallskip

Let $\mathcal{L}=(e(i,j))$ be a symmetric Latin square of side $m$ on the integers set $[m]$.
\smallskip

For every pair $i, j$ choose an arbitrary bijection 
$$\sigma_{i,j} : \mathcal{B}_i^{e(i,j)} \rightarrow \mathcal{B}_j^{e(j,i)}.$$ 

We require that $\sigma_{i,j}=\sigma_{j,i}^{-1}$.

\begin{construction}\label{Con1}
Let $\Gamma$ be a graph defined as follows:
\begin{itemize}
    \item The vertex set of $\Gamma$ is  $\displaystyle V=\bigcup_{i=1}^{m} \mathcal{P}_i.$ 
    \item Two different vertices $x\in \mathcal{P}_i$ and $y\in \mathcal{P}_j$ are adjacent in $\Gamma$ if and only if $$y \notin \sigma_{ij}(B_i^{e(i,j)}(x))\quad \mathrm{for\, all} \quad i,j\in [m].$$ 
\end{itemize}
\end{construction}

\begin{theorem}\label{Th1}  
If $\Gamma$ is a graph from Construction~\ref{Con1}, then  $\Gamma$ is a divisible design graph with parameters  
$$v = q^d (q^d - 1)/(q-1),\quad k = q^{d-1}(q^d - 1),$$
$$\lambda_1 = q^{d-1}(q^d - q^{d-1} - 1),\quad \lambda_2 = q^{d-2}(q-1)(q^d - 1),$$ $$m = (q^d - 1)/(q-1),\quad n = q^d.$$
Moreover, $\Gamma$ has four distinct eigenvalues  
$$\{q^d(q^{d-1} - 1),\,  q^{d-1},\,  0,\,  -q^{d-1}\}.$$
\end{theorem}
\begin{proof}
Let $\Gamma$ be a graph from Construction~\ref{Con1}. The number of vertices $v$ is equal to 
$q^d m = q^d(q^d - 1)/(q-1)$.

If $x$ is a vertex of $\Gamma$ belonging to $\mathcal{P}_i$, then
 $$\Gamma(x)= \bigcup_{j=1}^{m} (\mathcal{P}_j\setminus \sigma_{ij}(B_i^{e(i,j)}(x))).$$
Clearly, $$|\mathcal{P}_j\setminus \sigma_{ij}(B_i^{e(i,j)}(x)|=
|\mathcal{P}_j|-|\sigma_{ij}(B_i^{e(i,j)}(x)|= q^d - q^{d-1}.$$
Hence, $\Gamma$ is a regular graph of degree 
$$k = (q^d - q^{d-1})(q^d - 1)/(q-1) = q^{d-1}(q^d - 1).$$

Let $x$ and $y$ be two different vertices in $\Gamma$, belonging to the same class  $\mathcal{P}_i$. There are exactly $\lambda = (q^{d-1}-1)/(q-1)$ blocks in $\mathcal{D}_i$
containing $x$ and $y$. Thus, there are $(q^{d-1}-1)/(q-1)$ classes $\mathcal{P}_j$ in which $\sigma_{ij}(B_i^{e(i,j)}(x))$ and $\sigma_{ij}(B_i^{e(i,j)}(y))$ are the same blocks. Hence, $x$ and $y$ have exactly $q^d - q^{d-1}$ common neighbours in each of these classes. In the remaining classes $\sigma_{ij}(B_i^{e(i,j)}(x))$ and $\sigma_{ij}(B_i^{e(i,j)}(y))$  are disjoint blocks. Hence, $x$ and $y$ have exactly $q^d - 2q^{d-1}$ common neighbours in each of the remaining classes. 
Therefore, the number of common neighbours for $x$ and $y$ equals 
$$(q^d - q^{d-1}) \frac{(q^{d-1}-1)}{(q-1)} + (q^d - 2q^{d-1})(\frac{(q^d-1)}{(q-1)} - \frac{(q^{d-1}-1)}{(q-1)})= $$ $$ = q^{d-1}(q^d - q^{d-1} - 1).$$

Let $x$ be in $\mathcal{P}_i$, and $y$ be in $\mathcal{P}_j$, where $i\neq j$. 
In this case, for each $h\in [m]$ there are two blocks 
$$\sigma_{ih}(B_i^{e(i,h)}(x))\quad \mathrm{and}\quad \sigma_{jh}(B_j^{e(j,h)}(y))$$ in $\mathcal{D}_h$ which have exactly $q^{d-2}$ points in common.
 Thus, $x$ and $y$ have exactly $q^d - 2q^{d-1}+q^{d-2}$ common neighbours in each of $m=(q^d - 1)/(q-1)$ classes. Hence, the number of common neighbours for $x$ and $y$ equals 
$$(q^d - 2q^{d-1}+q^{d-2})\frac{(q^{d}-1)}{(q-1)} = q^{d-2}(q-1)(q^d - 1).$$ 
\end{proof}\hfill $\square$
\medskip

Remark that J.~Guo, K.~Wang and F. Li in \cite[Theorem 3.3]{GWL} obtained divisible design graphs  based on symplectic spaces with parameters $$\frac{q^{\nu +s}(q^{\nu -s} -1)}{(q-1)}, q^{\nu +s-1}(q^{\nu -s} -1), 
q^{\nu +s-2}(q^{\nu -s} -1)(q-1), q^{\nu +s-1}(q^{\nu -s} -q^{\nu -s} -1).$$
These parameters coincide with the parameters in Theorem \ref{Th1} whenever $s=0$.
\medskip 

We keep all the notation introduced in this section until the end of the article.

\section{Second construction}

Let $\Gamma^{\ast}$ be a divisible design graph from Construction~\ref{Con1} with parameters $(v^{\ast},k^{\ast},\lambda_1^{\ast},\lambda_2^{\ast},m^{\ast},n^{\ast})$ which was created using the Latin square $\mathcal{L^{\ast}}=(\tau^{\ast}(i,j))$.
\medskip

Let $\mathcal{L}=(e(i,j))$ be the square derived from the Latin square $\mathcal{L}^{\ast}=(\tau^{\ast}(i,j))$ by removing the row $h$ and the column $h$ for some $h$ from $[m^{\ast}]$. Since $\tau^{\ast}(i,h)$ and $\tau^{\ast}(h,i)$ are absent in $\mathcal{L}$ we can use a sort of randomness for $\mathcal{L}$ changing any $e(i,i)$ by $\tau^{\ast}(i,h)$ for $i< h$ and $e(i,i)$ by $\tau^{\ast}(i+1,h)$  for $i\geq h$. Thus, for each $h$ there are $2^{m^{\ast}-1}$ possibilities for choosing $\mathcal{L}$. 

\begin{construction}\label{Con2}
Let $\Gamma$ be a graph defined as follows:
\begin{itemize}
    \item The vertex set of $\Gamma$ is $\displaystyle V = V^{\ast}\setminus \mathcal{P}_h.$ 
   \item Two different vertices $x\in \mathcal{P}_i$ and $y\in \mathcal{P}_j$ are adjacent in $\Gamma$ if and only if $$y \notin \sigma_{ij}(B_i^{e(i,j)}(x))\quad \mathrm{for\, all} \quad i,j\in [m^{\ast}-1].$$ 
\end{itemize}
\end{construction}

\begin{theorem}\label{Th2}  
If $\Gamma$ is a graph from Construction~\ref{Con2}, then  $\Gamma$ is a divisible design graph with parameters  
$$v = q^{d+1} (q^{d-1} - 1)/(q-1),\quad k = q^d(q^{d-1} - 1),$$
$$\lambda_1 = q^d(q^{d-1} - q^{d-2} - 1),\quad \lambda_2 = q^{d-1}(q-1)(q^{d-1} - 1),$$ 
$$m = q^2(q^{d-1} - 1)/(q-1),\quad n = q^{d-1}.$$
Moreover, $\Gamma$ has four distinct eigenvalues  
$$\{q^d(q^{d-1} - 1),\,  q^{d-1},\,  0,\,  -q^{d-1}\}.$$
\end{theorem}
\begin{proof}
Let $\Gamma$ be a graph from Construction~\ref{Con2}. The number of vertices $v$ is equal to $$v^{\ast}-q^d = q^{d+1} \frac{(q^{d-1} - 1)}{(q-1)}.$$
%= q^d\left(\frac{(q^d - 1)}{(q-1)}-1\right)
If $x$ is a vertex of $\Gamma$ belonging to $\mathcal{P}_i$, then
 there are $k^{\ast}-(q^d - q^{d-1})$ vertices in $\Gamma(x)$.
Hence, $\Gamma$ is a regular graph of degree $$k=(q^d - q^{d-1})(m^{\ast}-1) = q^d(q^{d-1} - 1).$$
%(q^d - q^{d-1})(\frac{(q^d - 1)}{(q-1)}-1) = 

Let $x$ and $y$ be two different vertices from $\mathcal{D}_i$ belonging to one block of $h$-th parallel class. %then there are exactly $\lambda = (q^{d-1}-1)/(q-1)$ blocks through $x$ and $y$ in $\mathcal{D}_i$. 
Since $\mathcal{P}_h$ is not in $\Gamma$ there are $\lambda -1 =(q^{d-1}-1)/(q-1)-1$ classes $\mathcal{P}_j$, where $\sigma_{ij}(B_i^{e(i,j)}(x))$ and $\sigma_{ij}(B_i^{e(i,j)}(y))$ are the same. %Moreover, $x$ and $y$ have exactly $q^d - q^{d-1}$ common neighbours in each of these classes. In all remaining classes, $\sigma_{ij}(B_i^{e(i,j)}(x))$ and $\sigma_{ij}(B_i^{e(i,j)}(y))$  are parallel blocks. Hence, they have exactly $q^d - 2q^{d-1}$ common neighbours in each of them. 
Hence, the difference from $\lambda^{\ast}$ and $\lambda$ equals $q^d - q^{d-1}$.
Therefore, the number of common neighbours of $x$ and $y$ in $\Gamma$ equals 
$$\lambda_1^{\ast}-(q^d - q^{d-1}) = q^d(q^{d-1} - q^{d-2} - 1).$$ 

Let $x$ and $y$ don't lie on one block of $h$-th parallel class in $\mathcal{P}_i$ or $x\in \mathcal{P}_i$ and $y\in \mathcal{P}_j$, $i\neq j$.
In both cases, $x$ and $y$ have exactly $q^d - 2q^{d-1}+q^{d-2}$ common neighbours in each part of $V$. 
Hence, the number of common neighbours of $x$ and $y$ is equal to $$\lambda_2^{\ast}-(q^d - 2q^{d-1}+q^{d-2})=q^{d-1}(q-1)(q^{d-1}r - 1).$$ 

Since $x$ and $y$ have the same number of common neighbours if and only if they belong to one block in $h$-th parallel class in $\mathcal{P}_i$ for all $i\in [m^{\ast}]\setminus \{h\}$, then $\Gamma$ is a divisible design graph with classes on all these blocks.
Therefore, $n=q^{d-1}$ and $m=v/n=q^2(q^{d-1} - 1)/(q-1)$.
\end{proof}\hfill $\square$

\section{Third construction}

Let $\mathcal{D}_1, \dots ,\mathcal{D}_{m+1}$ be arbitrary affine designs all with parameters $(q,q^{d-2})$, where $m=(q^d - 1)/(q-1)$ is the number of parallel classes of blocks in each $\mathcal{D}_i$. For all  $i\in [m+1]$, let $\mathcal{D}_i=(\mathcal{P}_i, \mathcal{B}_i)$. Parallel classes in each $\mathcal{D}_i$ are enumerated by integers from $[m]$.
\medskip

Let $\mathcal{L}=(e(i,j))$ be a symmetric Latin square of side $m+1$ on the integers set $[m+1]$.

\begin{construction}\label{Con3}
Let $\Gamma$ be a graph defined as follows:
\begin{itemize}
    
    \item The vertex set of $\Gamma$ is  $\displaystyle V=\bigcup_{i=1}^{m+1} \mathcal{P}_i.$
    \item Two different vertices $x\in \mathcal{P}_i$ and $y\in \mathcal{P}_j$ are adjacent in $\Gamma$ if and only if $$y \notin \sigma_{ij}(B_i^{e(i,j)}(x))\quad \mathrm{for\, all} \quad i,j\in [m+1].$$ If $e(i,j)=m+1$, then there are no edges between $\mathcal{P}_i$ and $\mathcal{P}_j$.
\end{itemize}
\end{construction}
\begin{theorem}\label{Th3}  
Let $\Gamma$ be a graph from Construction~\ref{Con3}. If $q\neq 2$, then  $\Gamma$ is a divisible design graph with parameters 
$$v = q^d(q^d +q-2)/(q-1),\quad k = q^{d-1}(q^d - 1),$$
$$\lambda_1 = q^{d-1}(q^d - q^{d-1} - 1),\quad \lambda_2 = q^{d-1}(q - 1)(q^{d-1} - 1),$$ 
$$m = (q^d +q-2)/(q-1),\quad n = q^d.$$
Moreover, $\Gamma$ has five distinct eigenvalues  
$$\{q^{d-1}(q^d - 1),\, (q-1)q^{d-1},\,  q^{d-1},\, -q^{d-1},\, -(q-1)q^{d-1}\}.$$
If $q=2$, then $\Gamma$ is strongly regular graph with parameters
$$(2^{2d}, 2^{2d-1} - 2^{d-1}, 2^{2d-2} - 2^{d-1}, 2^{2d-2} - 2^{d-1}).$$
\end{theorem}
\begin{proof}
Let $\Gamma$ be a graph from Construction~\ref{Con3}. The number of vertices $v$ is equal to 
$$ q^d ( m+1) = q^d(q^d +q-2)/(q-1).$$
%q^d( (q^d - 1)/(q-1)+1)= 
If $x$ is a vertex of $\Gamma$ belonging to $\mathcal{P}_i$, then
 $$\Gamma(x)= \bigcup_{j=1,\, e(i,j)\neq m+1}^{m+1} (\mathcal{P}_j\setminus \sigma_{ij}(B_i^{e(i,j)}(x))).$$
Hence, $\Gamma$ is a regular graph of degree 
$$k = (q^d - q^{d-1})(q^d - 1)/(q-1) = q^{d-1}(q^d - 1).$$

Let $x$ and $y$ be two different vertices belonging to  $\mathcal{P}_i$. Since there are no edges between $\mathcal{P}_i$ and $\mathcal{P}_j$ if  $e(i,j)=m+1$, the number of common neighbours for $x$ and $y$ equals $q^{d-1}(q^d-q^{d-1} - 1)$ exactly the same as in Theorem~\ref{Th1}.

Let $x$ be in $\mathcal{P}_i$, and $y$ be in $\mathcal{P}_j$, where $i\neq j$. 
In this case, $x$ and $y$ have exactly $q^d - 2q^{d-1}+q^{d-2}$ common neighbours in each class $\mathcal{P}_h$, where $e(i,h) \neq m+1$ or $e(h,j) \neq m+1$. 
 Hence, the number of common neighbours for $x$ and $y$ equals 
$$(q^d - 2q^{d-1}+q^{d-2})\left(\frac{(q^{d}-1)}{(q-1)} - 1 \right)=q^{d-1}(q - 1)(q^{d-1} - 1).$$ 
If $q=2$, then $\lambda_1 = \lambda_2$ and $\Gamma$ is strongly regular graph.
\end{proof}\hfill $\square$

\section{Fourth construction}

Let $\mathcal{D}_1, \dots ,\mathcal{D}_{m+1}$ be arbitrary affine designs all with parameters $(q,q^{d-2})$, where $m=(q^d - 1)/(q-1)$ is the number of parallel classes of blocks in each $\mathcal{D}_i$. For all  $i\in [m+1]$, let $\mathcal{D}_i=(\mathcal{P}_i, \mathcal{B}_i)$. Parallel classes in each $\mathcal{D}_i$ are enumerated by integers from $[m]$.
\medskip

Let $\mathcal{L}=(e(i,j))$ be a symmetric Latin square of side $m+1$ and
there are no entries $m+1$ on the main diagonal.

\begin{construction}\label{Con4}
The graph $\Gamma$ is defined as follows:
\begin{itemize}
    
    \item The vertex set of $\Gamma$ is  $\displaystyle V=\bigcup_{i=1}^{m+1} \mathcal{P}_i.$
    \item Two vertices $x\in \mathcal{P}_i$ and $y\in \mathcal{P}_j$ are adjacent in $\Gamma$ if and only if $$y \notin \sigma_{ij}(B_i^{e(i,j)}(x))\quad \mathrm{for\, all} \quad i,j\in [m+1].$$ If $\ell(i,j)=m+1$, then every vertex from $\mathcal{P}_i$ is adjacent to all vertices from $\mathcal{P}_j$.
\end{itemize}
\end{construction}
\begin{theorem}\label{Th4}  
If $\Gamma$ is a graph from Construction~\ref{Con4}, then  $\Gamma$ is a strongly regular graph with parameters $(v, k, \lambda, \mu)$, where 
$$v = q^d(q^d +q-2)/(q-1),\quad k = q^{d-1}(q^d + q - 1),$$
$$\lambda =  \mu = q^{d-1}(q - 1)(q^{d-1} + 1).$$ 
\end{theorem}
\begin{proof}
Let $\Gamma$ be a graph from Construction~\ref{Con4}. 
If $x$ is a vertex from $\mathcal{P}_i$, then there exists the only $h\in [m+1]$ such that
$e(i,j) = m+1$. Hence, $$\Gamma(x)=\mathcal{P}_h\cup \bigcup_{j=1,\, j\neq h}^{m+1} (\mathcal{P}_j\setminus \sigma_{ij}(B_i^{e(i,j)}(x))).$$
Therefore, $\Gamma$ is a regular graph of degree 
$$k=q^d+(q^d - q^{d-1})(q^d - 1)/(q-1) =q^{d-1}(q^d + q - 1).$$

Let $x$ and $y$ be two different vertices in $\Gamma$, belonging to the same class  $\mathcal{P}_i$. 
Since every vertex from $\mathcal{P}_i$ is adjacent to all vertices from $\mathcal{P}_j$ in the case $e(i,j)=m+1$ then the number of common neighbours for $x$ and $y$ equals 
$$q^d + q^{d-1}(q^d-q^{d-1} - 1)=q^{d-1}(q-1)(q^{d-1} +1).$$

Let $x$ be in $\mathcal{P}_i$, and $y$ be in $\mathcal{P}_j$, where $i\neq j$. 
In this case, $x$ and $y$ have exactly $q^d - 2q^{d-1}+q^{d-2}$ common neighbours in each class $\mathcal{P}_h$, where $e(i,h) \neq m+1$ or $e(h,j) \neq m+1$ and $q^d-q^{d-1}$ common neighbours in $\mathcal{P}_h$, where $e(i,h) = m+1$ or $e(h,j) = m+1$. 
 Hence, the number of common neighbours for $x$ and $y$ equals 
$$2(q^d - q^{d-1}) + ((q^d - 2q^{d-1}+q^{d-2})\left(\frac{(q^{d}-1)}{(q-1)} - 1 \right)=$$
$$=q^{d-1}(q - 1)(q^{d-1} + 1).$$ 
Thus, $\Gamma$ is strongly regular graph.
\end{proof}\hfill $\square$

\begin{remark}
The complement of a graph from Construction \ref{Con4} is a strongly regular graph with  parameters 
$$(q^d\frac{(q^d +q-2)}{(q-1)},\, q^{d-1}\frac{( q^d  - 1))}{(q-1)}-1,\, q^{d-1}\frac{(q^{d-1} - 1)}{(q-1)} -2,\, q^{d-1}\frac{(q^{d-1} - 1)}{(q-1)} ).$$
These parameters were first obtained in Construction B  of type $S6$ in the paper \cite{MM} by  M.~Muzychuk. We were unable to check if the complement of any graph from Construction~\ref{Con4} can be obtained from Construction~S6.
\end{remark}

\section{Hadamard matrices and p-rank}

Let $J$ is the all ones matrix and $I$ is the identity matrix.

A square $(+1,-1)$-matrix $H$ is called a {\em Hadamard matrix} of order $n$ whenever $HH^T = nI$. A Hadamard matrix $H$ is called
{\em graphical} if $H$ is symmetric and it has constant diagonal, and $H$ is {\em regular} if all row and column sums are equal. 

In Section 4 of the paper \cite{AH} we find the following statement.
If $H$ is a graphical regular Hadamard matrix with the diagonal entries are equal to $1$,
then $\cfrac{1}{2} (J - H)$ is the adjacency matrix of 
a strongly regular graph with parameters 
$$(n,\, \cfrac{n}{2}\mp \cfrac{\sqrt{n}}{2},\, \cfrac{n}{4}\mp 
\cfrac{\sqrt{n}}{2},\, \cfrac{n}{4}\mp \cfrac{\sqrt{n}}{2}).$$ 
Conversely, any strongly regular graph with one of two these tuples of parameters comes from a Hadamard matrix in the described way.

If $n=2^{2d}$, then strongly regular graphs have parameters 
$$(2^{2d}, 2^{2d-1} \mp 2^{d-1}, 2^{2d-2} \mp 2^{d-1}, 2^{2d-2} \mp 2^{d-1}).$$
Thus, Theorem \ref {Th3} and Theorem \ref {Th4} are proposed as another point of view on these graphs in the even case.

One of the well-studied and useful parameters of a graph is the $p$-rank, that is, the rank of its adjacency matrix over the field with $p$-elements~\cite{BvE}. The $p$-rank can be useful to distinguish divisible design graphs and strongly regular graphs with the same parameters and spectrum \cite{AH}.

Variation of the numbering of parallel classes of blocks in the designs for all our constructions does not change the parameters and spectrum of the adjacency matrix of the resulting graph. Nevertheless, such variation can change the $2$-rank, in which case the new graph is obviously not isomorphic to the original one. %Some examples in the next sections confirm this fact.

\section{Examples of divisible design graphs from Construction~\ref{Con1},  \ref{Con2} and  \ref{Con3}}

All feasible parameter and known divisible design graphs with at most 27 vertices are given in~\cite{CH, HKM}. Divisible design graphs were studied in master's thesis by   M.A.~Meulenberg in~\cite{M2008}, where one can find feasible parameters of such graphs up to 50 vertices.
All divisible design graphs with at most 39 vertices, except for three tuples of parameters: $(32,15,6,7,4,8)$, $(32,17,8,9,4,8)$, $(36,24,15,16,4,9)$ were found by D.I.~Panasenko and L.V.~Shalaginov in~\cite{PSh}. 
Due to the large number of examples in this and the following  sections, we will only focus on examples derived from reduced Latin squares.

\subsection{q=2, d=2} 

By Theorem~\ref{Th1} and Theorem~\ref{Th2}  divisible design graphs have parameters $(12,6,2,3,3,4)$ and $(8,4,0,2,4,2)$, respectively. 
Only line graph of the octahedron has the first one, and $K_4\times K_2$ has the second one.

\subsection{q=3, d=2}

 By Theorem~\ref{Th1} and Theorem~\ref{Th2}, parameters of divisible design graphs are $(36,24,15,16,4,9)$ and $(27,18,9,12,18,3)$, respectively. 
There are two symmetric Latin squares of side $4$, which are the multiplication tables for $C_4$ the cyclic group of order $4$ and the Klein four-group. These Latin squares give two divisible design graphs with parameters $(36,24,15,16,4,9)$. Both of these graphs are known as Cayley graphs from~\cite[Example 4.2]{KSh}.

There are two known distant-regular graphs with intersection array	$\{8,6,1;1,3,8\}$.
The complement of these graphs are divisible design graphs with parameters  $(27,18,9,12,18,3)$.
We can obtain both of these graphs using Theorem~\ref{Th2} as well. 

\subsection{q=2, d=3} 

By Theorem~\ref{Th1}  divisible design graphs have parameters $(56,28,14,12,7,8)$. 
There are seven symmetric Latin squares of side $7$ (see for example Table 1.18 in Part III~\cite{CDW}).
We use $7$ copies of a $3$-dimensional point-hyperplane affine design over a finite field of order $2$ and identity for all bijections. For each parallel class of blocks take adjacency matrix of the complete multipartite graph on the vertex set of the design with parts on parallel blocks. We get seven matrices enumerated by parallel classes. Changing any entries of Latin squares by the matrix with the number equals of the entry we have got seven adjacency matrices of divisible design graphs  with parameters 
$(56,28,14,12,7,8)$. We used the computer system GAP with Package GRAPE \cite{GAP, GRAPE} to found these graphs. 

Divisible design graphs from the second and fourth Latin squares, as well as the sixth and seventh one are isomorphic. Hence, we obtained $5$  non-isomorphic divisible design graphs with  parameters $(56,28,14,12,7,8)$.

Any variation of the numbering of parallel classes in the starting design often gives non-isomorphic graphs. 

By Theorem~\ref{Th2} we have parameters $(48,24,8,12,12,4)$ and obtained lots of non-isomorphic divisible design graphs with such parameters. 

By Theorem~\ref{Th3} we have parameters $(45, 24, 15, 12, 5, 9)$. There are one symmetric Latin square of side $5$ (see for example Table 1.18 in Part III~\cite{CDW}). The graph we obtained using this Latin square from Construction~\ref{Con3} has the automorphism group $(C_3 \times C_3) : C_8$.

All graphs with parameters $(56,28,14,12,7,8)$, $(48,24,8,12,12,4)$  and $(45, 24, 15, 12, 5, 9)$  seem to be new and are available on the website \cite{DezaGraphs}.

\section{Examples of strongly regular graphs from Construction \ref{Con3} and \ref{Con4} }

\subsection{q=2, d=3} 

A.~Abiad, S. Butler and W.H.~Haemers in \cite{AH} and F.~Ihringer \cite{FI} provided an abundance of 
strongly regular graphs with parameters $(64,28,12,12)$. 
In the libraries of small loops in GAP package LOOPS \cite{LOOPS} we found symmetric Latin squares as multiplication tables of abelian loops of order 8. We used the multiplication tables of three abelian groups of order $8$ as well. Using these Latin squares we obtained  fifteen non-isomorphic graphs with parameters $(64,28,12,12)$ from Construction~\ref{Con3}.
The following table is contain the results, where the first column gives the $2$-rank, and the second column gives the structure description of the automorphism group.  
\medskip 

{\small 
    \begin{tabular}{ll}
   8  &  $(C_2 \times C_2 \times C_2 \times C_2 \times C_2 \times C_2) : S_8$;\\
   8  &  $((((((C_2 \times C_2 \times (((C_2 \times C_2 \times C_2 \times C_2) : C_2) : C_2)) : C_2) : C_2) : C_2) : C_2) : C_3) : C_2$;\\
 10  &  $C_2 \times C_2 \times C_2 \times D_8$;\\
  10  & $C_2 \times C_2 \times (((C_2 \times C_2 \times C_2 \times C_2) : C_3) : C_2)$;\\
  10  & $C_2 \times (((((C_2 \times ((C_2 \times C_2 \times C_2 \times C_2) : C_2)) : C_2) : C_3) : C_2) : C_2)$;\\
  12  &   $C_2 \times C_2 \times C_2 \times D_8$;\\
  12  &   $C_2 \times C_2 \times C_2 \times C_2 \times C_2$;\\
   12   &  $C_2 \times D_8 \times D_8$;\\
  12  &  $C_2 \times ((C_2 \times ((C_4 \times C_2) : C_2)) : C_2)$;\\
   12  &    $C_2 \times C_2 \times C_2 \times C_2 \times C_2$;\\
  12    &   $C_2 \times C_2 \times C_2 \times D_8$;\\
  12   &  $(C_2 \times C_2 \times C_2 \times D_8) : C_2$;\\
  12  &  $(C_4 \times C_2 \times C_2 \times C_2 \times C_2) : (C_2 \times A_5)$;\\ 
  14    &    $C_2 \times C_2 \times C_2 \times C_2 \times C_2$;\\
  14    &   $C_2 \times C_2 \times C_2 \times C_2 \times C_2$.\\
    \end{tabular}}
 \medskip 

 We have one graph with the automorphism group of order $7680$ and the structure description $(C_4 \times C_2 \times C_2 \times C_2 \times C_2) : (C_2 \times A_5)$.  This graph was created using one of conjugacy closed loops from the CC library of GAP package LOOPS~\cite{LOOPS}. Such a graph was not found in \cite{FI}.
 
 \subsection{q=3, d=3} 
 
Strongly regular graphs with  parameters $(378, 261, 180, 180)$ were first obtained from Construction B  of type $S6$ in the paper \cite{MM} by  M.~Muzychuk. 
By Construction~\ref{Con4} we found three strongly regular graphs with such parameters. 
There are three symmetric Latin square of side $14$. First one is the Cayley table of the cyclic group  $C_{14}$ and two others are the Cayley table of Steiner loops in the library of GAP package LOOPS \cite{LOOPS}.  Using this Latin squares we obtained three non-isomorphic strongly regular graphs with  parameters $(378, 261, 180, 180)$. Their $3$-ranks and the structural descriptions of the automorphism groups are given in the following table.
\medskip

{\small 
    \begin{tabular}{cl}
66 & $(C_3 \times C_3 \times C_3) : C_2$;\\ 

65 & $C_2 \times ((C_3 \times C_3 \times C_3) : C_2)$;\\ 

65 & $((A_9 \times A_9  \times A_9) : (C_2 \times S_4)$.\\
    \end{tabular}}
\medskip 

All graphs with parameters $(64,28,12,12)$ and $(378, 261, 180, 180)$ which were found in this section are available on the website \cite{DezaGraphs}.

%\section*{Acknowledgements}  

\end{document}